\documentclass[12pt,lenq]{amsart}
\usepackage{amsmath}
\usepackage{amscd}
\usepackage{amssymb}
\usepackage{amsbsy}
\usepackage{amsfonts}
\usepackage{latexsym}
\usepackage{graphics}
\usepackage{graphicx}
\usepackage{amsmath,amscd,latexsym}
\usepackage{multirow}
\usepackage{array}
\usepackage{paralist}
\usepackage{titletoc}

\theoremstyle{plain}
\newtheorem{theorem}{Theorem}[section]

\newtheorem{lemma}[theorem]{Lemma}
\newtheorem{corollary}[theorem]{Corollary}
\newtheorem{remark}[theorem]{Remark}
\newtheorem{definition}[theorem]{Definition}
\newtheorem{notation}[theorem]{Notation}

\newtheorem{main theorem}[theorem]{Main Theorem}

\newtheorem{question}[theorem]{Question}

\newcommand{\ZZ}{\mathbb{Z}}
\newcommand{\QQ}{\mathbb{Q}}
\newcommand{\RR}{\mathbb{R}}

\newcommand{\HH}{\mathbb{H}}
\newcommand{\QQQ}{\hat{\mathbb{Q}}}

\newcommand{\PConway}{\mbox{\boldmath$S$}}

\newcommand{\rtangle}[1]{(B^3,t({#1}))}

\newcommand{\DD}{\mathcal{D}}
\newcommand{\RGPC}[2]{\Gamma({#1};{#2})}

\newcommand{\RGP}[1]{\Gamma_{#1}}

\newcommand{\Hecke}{\mbox{$G$}}

\newcommand{\orbs}{\mbox{\boldmath$S$}}

\newcommand{\svert}{\,|\,}

\newcommand{\llangle}{\langle\langle}
\newcommand{\rrangle}{\rangle\rangle}

\newcommand{\lp}{(\hskip -0.07cm (}
\newcommand{\rp}{)\hskip -0.07cm )}

\makeatletter
\renewcommand\subsection{\@startsection{subsection}{2}{0mm}
    {-10.5dd plus-8pt minus-4pt}{10.5dd}
     {\normalsize\upshape}}
\makeatother

\begin{document}

\title{Simple loops on 2-bridge spheres in Heckoid orbifolds for the trivial knot}

\author{Donghi Lee}
\address{Department of Mathematics\\
Pusan National University\\
San-30 Jangjeon-Dong, Geumjung-Gu, Pusan, 609-735, Republic of Korea}
\email{donghi@pusan.ac.kr}

\author{Makoto Sakuma}
\address{Department of Mathematics\\
Graduate School of Science\\
Hiroshima University\\
Higashi-Hiroshima, 739-8526, Japan}
\email{sakuma@math.sci.hiroshima-u.ac.jp}

\keywords{$2$-bridge link group, even Heckoid group, Hecke group,
van Kampen diagram, annular diagram}

\subjclass[2010]{Primary 20F06, 57M25\\
\indent {The first author was supported by a 2-Year Research Grant of
Pusan National University.}}

\begin{abstract}
In this paper,
we give a necessary and sufficient condition
for an essential simple loop on a $2$-bridge sphere
in an even Heckoid orbifold for the trivial knot
to be null-homotopic, peripheral or torsion in the orbifold.
We also give a necessary and sufficient condition
for two essential simple loops on a $2$-bridge sphere
in an even Heckoid orbifold for the trivial knot
to be homotopic in the orbifold.
\end{abstract}
\maketitle


\section{Introduction}
In \cite{lee_sakuma_6}, following Riley's work~\cite{Riley2},
we introduced the {\it Heckoid group $\Hecke(r;n)$ of index $n$ for
a $2$-bridge link, $K(r)$, of slope $r \in \QQ$},
as the orbifold fundamental group
of the {\it Heckoid orbifold $\orbs(r;n)$ of index $n$ for $K(r)$}.
Here $n$ is an integer or a half-integer greater than $1$.
The Heckoid group and the Heckoid
orbifold are said to be {\it even} or {\it odd}
according to whether $n$ is an integer or a half-integer.
When $K(r)$ is the trivial knot
and $n$ is an integer greater than $1$,
the even Heckoid orbifold
$\orbs(r;n)\cong \orbs(0;n)$ is as illustrated in Figure~\ref{fig.trivial_orbifold},
and the even Heckoid group $\Hecke(r;n)\cong \Hecke(0;n)$ is isomorphic to
the index $2$ subgroup $\langle P, SPS^{-1}\rangle$
of the classical {\it Hecke group}
$\langle P, S\rangle$ introduced in \cite{Hecke} (cf. \cite[Remark~2.5]{lee_sakuma_7}), where
\[
P=\begin{pmatrix}
1 & \quad 2\cos\frac{\pi}{2n}\\
0 & \quad 1
\end{pmatrix},
\quad
S=
\begin{pmatrix}
0 & \quad 1\\
-1 & \quad 0
\end{pmatrix}.
\]

\begin{figure}[htbp]
\begin{center}
\includegraphics{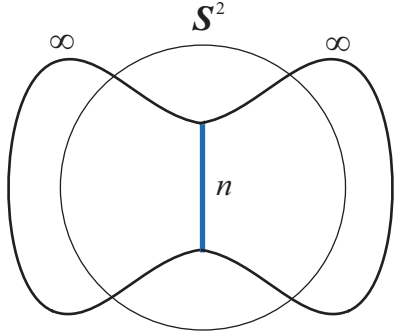}
\end{center}
\caption{\label{fig.trivial_orbifold}
The even Heckoid orbifold $\orbs(0;n)$
of index $n$ for the trivial knot}
\end{figure}

In \cite[Theorem~2.3]{lee_sakuma_7},
we gave a complete characterization of
those essential simple loops on a $2$-bridge sphere
in an even Heckoid orbifold $\orbs(r;n)$
which are null-homotopic in $\orbs(r;n)$.
Furthermore, in a series of papers \cite{lee_sakuma_9, lee_sakuma_10},
we gave a necessary and sufficient condition
for two essential simple loops on a $2$-bridge sphere
in $\orbs(r;n)$ to be homotopic in $\orbs(r;n)$,
and a necessary and sufficient condition
for an essential simple loop on a $2$-bridge sphere
in $\orbs(r;n)$ to be peripheral or torsion in $\orbs(r;n)$.
However these results were obtained for the
generic case when $r$ is non-integral
and we deferred the results for the special
case when $r$ is integral (cf. \cite[Remark~2.5]{lee_sakuma_7}).
The purpose of this note is to obtain similar results
for the remaining special case
when $r$ is integral, that is, $K(r)$ is a trivial knot.

This paper is organized as follows.
In Section~\ref{sec:main_result},
we describe our main result.
In Section~\ref{sec:preliminaries}, we recall
the upper presentation of an even Heckoid group,
over which we introduce van Kampen diagrams and annular diagrams.
Finally, Section~\ref{sec:proof_of_main_theorem}
is devoted to the proof of Theorem~\ref{thm:main_theorem}.

\section{Main result}
\label{sec:main_result}

We quickly recall notation and basic facts introduced in \cite{lee_sakuma_6}.
The {\it Conway sphere} $\PConway$ is the 4-times punctured sphere
which is obtained as the quotient of $\RR^2-\ZZ^2$
by the group generated by the $\pi$-rotations around
the points in $\ZZ^2$.
For each $s \in \QQQ:=\QQ\cup\{\infty\}$,
let $\alpha_s$ be the simple loop in $\PConway$
obtained as the projection of a line in $\RR^2-\ZZ^2$
of slope $s$.
We call $s$ the {\it slope} of the simple loop $\alpha_s$.

For each $r\in \QQQ$,
the {\it $2$-bridge link $K(r)$ of slope $r$}
is the sum of the rational tangle
$\rtangle{\infty}$ of slope $\infty$ and
the rational tangle $\rtangle{r}$ of slope $r$.
Recall that $\partial(B^3-t(\infty))$ and $\partial(B^3-t(r))$
are identified with $\PConway$
so that $\alpha_{\infty}$ and $\alpha_r$
bound disks in $B^3-t(\infty)$ and $B^3-t(r)$, respectively.
By van-Kampen's theorem, the link group $G(K(r))=\pi_1(S^3-K(r))$ is obtained as follows:
\[
G(K(r))=\pi_1(S^3-K(r))
\cong \pi_1(\PConway)/ \llangle\alpha_{\infty},\alpha_r\rrangle
\cong \pi_1(B^3-t(\infty))/\llangle\alpha_r\rrangle.
\]

On the other hand, if $r$ is a rational number and $n\ge 2$ is an integer,
then the even Heckoid orbifold $\orbs(r;n)$ contains a Conway sphere $\PConway$,
and the even Hekoid group $\Hecke(r;n)$,
which is defined as the orbifold fundamental group of $\orbs(r;n)$,
has the following description as the quotient of the fundamental group
of the Conway sphere $\PConway$ (see \cite[p.~242]{lee_sakuma_6}):
\[
\Hecke(r;n)
\cong\pi_1(\PConway)/ \llangle\alpha_{\infty},\alpha_r^n\rrangle
\cong \pi_1(B^3-t(\infty))/\llangle\alpha_r^n\rrangle.
\]

We are interested in the following naturally arising question.

\begin{question}
\label{question1}
{\rm
For $r$ a rational number and $n$ an integer greater than $1$,
consider the even Heckoid orbifold $\orbs(r;n)$ for the $2$-bridge link $K(r)$.
\begin{enumerate}[\indent \rm (1)]
\item
Which essential simple loop $\alpha_s$ on $\PConway$
is null-homotopic in $\orbs(r;n)$?
\item
For two distinct essential simple loops $\alpha_s$ and $\alpha_{s'}$ on $\PConway$,
when are they homotopic in $\orbs(r;n)$?
\item
Which essential simple loop $\alpha_s$ on $\PConway$
is peripheral or torsion in $\orbs(r;n)$?
\end{enumerate}
}
\end{question}

This question originated from Minsky's question~\cite[Question~5.4]{Gordon},
and was completely solved in the series of papers
\cite{lee_sakuma_6, lee_sakuma_7, lee_sakuma_9, lee_sakuma_10}
for the generic case when $r$ is non-integral, that is,
$K(r)$ is not a trivial knot.
See \cite{lee_sakuma_8} for an overview of these works.

We note that
(1) a loop in the orbifold $\orbs(r;n)$ is {\it null-homotopic} in $\orbs(r;n)$
if and only if it determines the trivial conjugacy class of the Heckoid group $\Hecke(r;n)$, and
(2) two loops in $\orbs(r;n)$ are {\it homotopic} in $\orbs(r;n)$
if and only if they determine the same conjugacy class in $\Hecke(r;n)$
(see \cite{BMP, Boileau-Porti} for the concept of homotopy in orbifolds).
We say that a loop in $\orbs(r;n)$ is {\it peripheral} if and only if it is homotopic to a loop
in the paring annulus naturally associated with $\orbs(r;n)$ (see \cite [Section~6]{lee_sakuma_6}),
i.e., it represents the conjugacy class of a power of a meridian of $\Hecke(r;n)$.
We also say that a loop in $\orbs(r;n)$ is {\it torsion} if it represents the conjugacy class
of a non-trivial torsion element of $\Hecke(r;n)$.
If we identify $\Hecke(r;n)$ with a Kleinian group generated by two parabolic transformations
(see \cite[Theorem~2.2]{lee_sakuma_6}),
then a loop $\orbs(r;n)$ is peripheral or torsion
if and only if it corresponds to a parabolic transformation or a non-trivial elliptic transformation
accordingly.
Thus Question~\ref{question1} can be interpreted as a question
on the even Heckoid group $\Hecke(r;n)$.

Let $\DD$ be the {\it Farey tessellation}
of the upper half plane $\HH^2$.
Then $\QQQ$ is identified with the set of the ideal vertices of $\DD$.
Let $\RGP{\infty}$ be the group of automorphisms of
$\DD$ generated by reflections in the edges of $\DD$ with an endpoint $\infty$.
For $r$ a rational number
and $n$ an integer or a half-integer greater than $1$,
let $C_r(2n)$ be the group of automorphisms of $\DD$ generated
by the parabolic transformation, centered on the vertex $r$,
by $2n$ units in the clockwise direction,
and let $\RGPC{r}{n}$ be the group generated by $\RGP{\infty}$ and $C_r(2n)$.
The answer to Question \ref{question1} obtained in
\cite{lee_sakuma_6, lee_sakuma_7, lee_sakuma_9, lee_sakuma_10},
for the general case when $r$ is non-integral,
is given in terms of the action of $\RGPC{r}{n}$ on $\partial \HH^2=\hat\RR$.
The answer to the remaining case when $r$ is an integer is
also given in a similar way.

Observe that, when $n\ge 2$ is an integer,
the group $\Gamma(0;n)$ is the free product of
three cyclic groups of order $2$ generated by the reflections
in the Farey edges $\langle \infty,0\rangle$, $\langle \infty,1\rangle$
and $\langle 0,1/n \rangle$
(see Figure~\ref{fig.fd_orbifold_trivial}).
In fact,
the region, $R$, of $\HH^2$ bounded by these three Farey edges is a fundamental domain
for the action of $\Gamma(0;n)$ on $\HH^2$.
Note that the intersection of the closure of $R$ with $\partial \HH^2$
is the disjoint union of the discrete set $\{\infty, 0\}$ and the closed interval
$I(0;n):=[1/n,1]$.
The following two theorems give a complete answer to
Question~\ref{question1} for the remaining special case.

\begin{figure}[htbp]
\begin{center}
\includegraphics{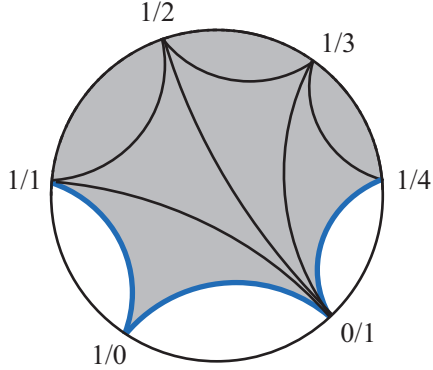}
\end{center}
\caption{\label{fig.fd_orbifold_trivial}
The fundamental domain of $\RGPC{0}{n}$ in the
Farey tessellation (the shaded domain) for $n=4$}
\end{figure}

\begin{theorem}
\label{thm:fundametal_domain}
Suppose that $n$ is an integer greater than $1$. 
Then for any $s\in\QQQ$, there is a unique rational number
$s_0\in I(0;n) \cup \{\infty, 0\}$
such that $s$ is contained in the $\RGPC{r}{n}$-orbit of $s_0$.
Moreover the conjugacy classes
$\alpha_s$ and $\alpha_{s_0}$ in $\Hecke(0;n)$ are equal.
In particular, (i) if $s_0=\infty$, then $\alpha_s$ is the trivial conjugacy class
in $\Hecke(0;n)$, and
(ii) if $s_0=0$, then $\alpha_s$ is torsion
in $\Hecke(0;n)$.
\end{theorem}

\begin{theorem}
\label{thm:main_theorem}
Suppose that $n$ is an integer greater than $1$.
Then the following hold.
\begin{enumerate}[\indent \rm (1)]
\item Any simple loop in $\{\alpha_s \svert s\in I(0;n) \cup \{0\}\}$
does not represent the trivial element of $\Hecke(0;n)$.

\item The simple loops $\{\alpha_s \svert s\in I(0;n)\}$
represent mutually distinct conjugacy classes in $\Hecke(0;n)$.

\item There is no rational number $s \in I(0;n)$
for which $\alpha_s$ is peripheral in $\Hecke(0;n)$.

\item There is no rational number $s \in I(0;n)$
for which $\alpha_s$ is torsion in $\Hecke(0;n)$.
\end{enumerate}
\end{theorem}

The proof of Theorem~\ref{thm:fundametal_domain}
is essentially the same as that of \cite[Theorem~2.2]{lee_sakuma_6}.
In fact, the first assertion is proved as in \cite[Lemma~7.1]{lee_sakuma}
by using the fact
that $R$ is a fundamental domain for the action of $\RGPC{r}{n}$ on $\HH^2$.
The second assertion
is nothing other than \cite[Theorem~2.4]{lee_sakuma_6}.
The last assertion follows immediately from the second assertion.

We shall prove
Theorem~\ref{thm:main_theorem}
with a classical geometric method in combinatorial group theory
such as using van Kampen diagrams and annular diagrams
over two-generator and one-relator presentations,
so-called the upper presentations, of even Heckoid groups.

\section{Preliminaries}
\label{sec:preliminaries}

\subsection{Upper presentations of even Heckoid groups}
\label{subsec:group_presentation}

We introduce the
upper presentation of an even Heckoid group $\Hecke(r;n)$,
where $r$ is a rational number and $n \ge 2$ is an integer.
Recall that
\[
\Hecke(r;n)\cong
\pi_1(\PConway)/ \llangle\alpha_{\infty},\alpha_r^n\rrangle
\cong
\pi_1(B^3-t(\infty))/ \llangle\alpha_r^n\rrangle.
\]
Let $\{a,b\}$ be the standard meridian generator pair of $\pi_1(B^3-t(\infty), x_0)$
as described in \cite[Section~3]{lee_sakuma}
(see also \cite[Section~5]{lee_sakuma_0}).
Then $\pi_1(B^3-t(\infty))$ is identified with the free group $F(a,b)$.
Obtain a word $u_r\in F(a,b)\cong\pi_1(B^3-t(\infty))$
which is represented by the simple loop $\alpha_r$.
It then follows that
\[
\Hecke(r;n) \cong\pi_1(B^3-t(\infty))/\llangle \alpha_r^n\rrangle
\cong \langle a, b \svert u_r^n \rangle.
\]
This two-generator and one-relator presentation
is called the {\it upper presentation} of the even Heckoid group $\Hecke(r;n)$.
It is known by \cite[Proposition~1]{Riley}
that there is a nice formula to find $u_r$ as follows.
(For a geometric description, see \cite[Section~5]{lee_sakuma_0}.)

\begin{lemma}
\label{presentation}
Let $p$ and $q$ be relatively prime integers
such that $p \ge 1$.
For $1 \le i \le p-1$, let
\[\epsilon_i = (-1)^{\lfloor iq/p \rfloor},\]
where $\lfloor x \rfloor$ is the greatest integer not exceeding $x$.
\begin{enumerate}[\indent \rm (1)]
\item If $p$ is odd, then
\[
u_{q/p}=a\hat{u}_{q/p}b^{(-1)^q}\hat{u}_{q/p}^{-1},\]
where
$\hat{u}_{q/p} = b^{\epsilon_1} a^{\epsilon_2} \cdots b^{\epsilon_{p-2}} a^{\epsilon_{p-1}}$.
\item If $p$ is even, then
\[
u_{q/p}=a\hat{u}_{q/p}a^{-1}\hat{u}_{q/p}^{-1},\]
where
$\hat{u}_{q/p} = b^{\epsilon_1} a^{\epsilon_2} \cdots a^{\epsilon_{p-2}} b^{\epsilon_{p-1}}$.
\end{enumerate}
\end{lemma}

\begin{remark}
\label{rem:epsilon}
{\rm
We have $u_{0}=ab$. Thus if $r$ is an integer, then
\[\Hecke(r;n) \cong \Hecke(0;n) \cong \langle a, b \svert (ab)^n \rangle.\]
}
\end{remark}

Now we define the cyclic sequence
$CS(r)$, which is read from $u_r$ defined in Lemma~\ref{presentation},
and review an important property of this sequence
from \cite{lee_sakuma}.
To this end we fix some definitions and notation.
Let $X$ be a set.
By a {\it word} in $X$, we mean a finite sequence
$x_1^{\epsilon_1}x_2^{\epsilon_2}\cdots x_n^{\epsilon_n}$
where $x_i\in X$ and $\epsilon_i=\pm1$.
Here we call $x_i^{\epsilon_i}$ the {\it $i$-th letter} of the word.
For two words $u, v$ in $X$, by
$u \equiv v$ we denote the {\it visual equality} of $u$ and
$v$, meaning that if $u=x_1^{\epsilon_1} \cdots x_n^{\epsilon_n}$
and $v=y_1^{\delta_1} \cdots y_m^{\delta_m}$ ($x_i, y_j \in X$; $\epsilon_i, \delta_j=\pm 1$),
then $n=m$ and $x_i=y_i$ and $\epsilon_i=\delta_i$ for each $i=1, \dots, n$.
For example, two words $x_1x_2x_2^{-1}x_3$ and $x_1x_3$ ($x_i \in X$) are {\it not} visually equal,
though $x_1x_2x_2^{-1}x_3$ and $x_1x_3$ are equal as elements of the free group with basis $X$.
The length of a word $v$ is denoted by $|v|$.
A word $v$ in
$X$ is said to be {\it reduced} if $v$ does not contain $xx^{-1}$ or $x^{-1}x$ for any $x \in X$.
A word is said to be {\it cyclically reduced}
if all its cyclic permutations are reduced.
A {\it cyclic word} is defined to be the set of all cyclic permutations of a
cyclically reduced word. By $(v)$ we denote the cyclic word associated with a
cyclically reduced word $v$.
Also by $(u) \equiv (v)$ we mean the {\it visual equality} of two cyclic words
$(u)$ and $(v)$. In fact, $(u) \equiv (v)$ if and only if $v$ is visually a cyclic shift
of $u$.

\begin{definition}
\label{def:alternating}
{\rm
(1) Let $(v)$ be a cyclic word in
$\{a, b\}$. Decompose $(v)$ into
\[
(v) \equiv (v_1 v_2 \cdots v_t),
\]
where all letters in $v_i$ have positive (resp., negative) exponents,
and all letters in $v_{i+1}$ have negative (resp., positive) exponents (taking
subindices modulo $t$). Then the {\it cyclic} sequence of positive integers
$CS(v):=\lp |v_1|, |v_2|, \dots, |v_t| \rp$ is called
the {\it cyclic $S$-sequence of $(v)$}.
Here the double parentheses denote that the sequence is considered modulo
cyclic permutations.

(2) A reduced word $v$ in $\{a,b\}$ is said to be {\it alternating}
if $a^{\pm 1}$ and $b^{\pm 1}$ appear in $v$ alternately,
i.e., neither $a^{\pm2}$ nor $b^{\pm2}$ appears in $v$.
A cyclic word $(v)$ is said to be {\it alternating}
if all cyclic permutations of $v$ are alternating.
In the latter case, we also say that $v$ is {\it cyclically alternating}.
}
\end{definition}

\begin{definition}
\label{def4.1(3)}
{\rm
For a rational number $s$ with $0<s\le 1$,
let $u_s$ be defined as in Lemma~\ref{presentation}.
Then the symbol $CS(s)$ denotes the cyclic $S$-sequence $CS(u_s)$ of $(u_s)$,
which is called the {\it cyclic S-sequence of slope $s$}).}
\end{definition}

We recall the following basic property of $CS(s)$.

\begin{lemma} [{\cite[Proposition~4.3]{lee_sakuma}}]
\label{lem:properties}
Suppose that $s$ is a rational number
with $0<s\le1$, and write $s$ as a continued fraction:
\begin{center}
\begin{picture}(230,70)
\put(0,48){$\displaystyle{
s=[m_1,m_2, \dots,m_k]:=
\cfrac{1}{m_1+
\cfrac{1}{ \raisebox{-5pt}[0pt][0pt]{$m_2 \, + \, $}
\raisebox{-10pt}[0pt][0pt]{$\, \ddots \ $}
\raisebox{-12pt}[0pt][0pt]{$+ \, \cfrac{1}{m_k}$}
}},}$}
\end{picture}
\end{center}
where $k \ge 1$, $(m_1, \dots, m_k) \in (\mathbb{Z}_+)^k$ and
$m_k \ge 2$ unless $k=1$. Then the following hold.
\begin{enumerate}[\indent \rm (1)]
\item Suppose $k=1$, i.e., $s=1/m_1$.
Then $CS(s)=\lp m_1,m_1 \rp$.

\item Suppose $k\ge 2$. Then each term of $CS(s)$ is either $m_1$ or $m_1+1$.
\end{enumerate}
\end{lemma}

\begin{corollary}
\label{cor:properties}
Suppose that $n$ is an integer greater than $1$. 
If $s$ is a rational number with $1/n \le s \le 1$,
then every term of $CS(s)$ less than or equal to $n$.
\end{corollary}

\begin{proof}
If $s=1/n$, then $CS(s)=\lp n, n \rp$ by Lemma~\ref{lem:properties}(1),
and hence the assertion clearly holds. So let $1/n< s \le 1$.
If $s=[m_1, \dots, m_k]$ is a continued fraction
as in the statement of Lemma~\ref{lem:properties},
then $m_1 \le n-1$. Hence by Lemma~\ref{lem:properties},
the assertion holds.
\end{proof}

\subsection{Van Kampen diagrams and annular diagrams}
\label{subsec:Van_Kampen_diagrams_and_annular_diagrams}

Let us begin with necessary definitions and notation following \cite{lyndon_schupp}.
A {\it map} $M$ is a finite $2$-dimensional cell complex
embedded in $\RR^2$.
To be precise, $M$ is a finite collection of vertices ($0$-cells), edges ($1$-cells),
and faces ($2$-cells) in $\RR^2$ satisfying the following conditions.
\begin{enumerate}[\indent \rm (i)]
\item A vertex is a point in $\RR^2$.

\item An edge $e$ is homeomorphic to an open interval
such that $\bar e=e\cup\{ a\}\cup \{b\}$,
where $a$ and $b$ are vertices of $M$ which are possibly identical.

\item For each face $D$ of $M$,
there is a continuous map $f$ from the
$2$-ball $B^2$ to $\RR^2$ such that
\begin{enumerate}
\item the restriction of $f$ to the interior of $B^2$
is a homeomorphism onto $D$, and

\item the image of $\partial B^2$ is equal to
$\cup_{i=1}^t \bar e_i$ for some set $\{e_1,\dots, e_t\}$ of edges of $M$.
\end{enumerate}
\end{enumerate}
The underlying space of $M$, i.e., the union of the cells in $M$,
is also denoted by the same symbol $M$.
The boundary (frontier), $\partial M$, of $M$ in $\RR^2$
is regarded as a $1$-dimensional subcomplex of $M$.
An edge may be traversed in either of two directions.
If $v$ is a vertex of a map $M$, $d_M(v)$, the {\it degree of $v$},
denotes the number of oriented edges in $M$ having $v$ as initial vertex.
A vertex $v$ of $M$ is called an {\it interior vertex}
if $v\not\in \partial M$, and an edge $e$ of $M$ is called
an {\it interior edge} if $e\not\subset \partial M$.

A {\it path} in $M$ is a sequence of oriented edges $e_1, \dots, e_t$ such that
the initial vertex of $e_{i+1}$ is the terminal vertex of $e_i$ for
every $1 \le i \le t-1$. A {\it cycle} is a closed path, namely
a path $e_1, \dots, e_t$
such that the initial vertex of $e_1$ is the terminal vertex of $e_t$.
If $D$ is a face of $M$, any cycle of minimal length which includes
all the edges of the boundary, $\partial D$, of $D$
going around once along the boundary of $D$
is called a {\it boundary cycle} of $D$.
To be precise it is defined as follows.
Let $f: B^2 \rightarrow D$ be a continuous map satisfying
condition (iii) above.
We may assume that $\partial B^2$ has a cellular structure
such that the restriction of $f$ to each cell is a homeomorphism.
Choose an arbitrary orientation of $\partial B^2$, and let
$\hat e_1, \dots, \hat e_t$ be the oriented edges of $\partial B^2$,
which are oriented in accordance with the orientation of $\partial B^2$
and which lie on $\partial B^2$ in this cyclic order with respect to the orientation of $\partial B^2$.
Let $e_i$ be the orientated edge $f(\hat e_i)$ of $M$.
Then the cycle $e_1, \dots, e_t$ is a boundary cycle of $D$.

Let $F(X)$ be the free group with basis $X$.
A subset $R$ of $F(X)$ is said to be {\it symmetrized},
if all elements of $R$ are cyclically reduced and, for each $w \in R$,
all cyclic permutations of $w$ and $w^{-1}$ also belong to $R$.

\begin{definition}
{\rm Let $R$ be a symmetrized subset of $F(X)$. An {\it $R$-diagram} is
a pair $(M,\phi)$ of
a map $M$ and a function $\phi$ assigning to each oriented edge $e$ of $M$, as a {\it label},
a reduced word $\phi(e)$ in $X$ such that the following hold.
\begin{enumerate}[\indent \rm (i)]
\item If $e$ is an oriented edge of $M$ and $e^{-1}$ is the oppositely oriented edge,
then $\phi(e^{-1})=\phi(e)^{-1}$.

\item For any boundary cycle $\delta$ of any face of $M$,
$\phi(\delta)$ is a cyclically reduced word
representing an element of $R$.
(If $\alpha=e_1, \dots, e_t$ is a path in $M$, we define $\phi(\alpha) \equiv \phi(e_1) \cdots \phi(e_t)$.)
\end{enumerate}
We denote an $R$-diagram $(M,\phi)$ simply by $M$.
}
\end{definition}

\begin{definition}
\label{def:annular_map}
{\rm
Let a group $G$ be presented by $G=\langle X \,|\, R \, \rangle$ with $R$ being symmetrized.

(1) A connected and simply connected $R$-diagram is called a {\it van Kampen diagram}
over $G=\langle X \,|\, R \, \rangle$.

(2) An $R$-diagram $M$ is called an {\it annular diagram} over
$G=\langle X \,|\, R \, \rangle$, if $\RR^2-M$ has exactly two connected components.
}
\end{definition}

Suppose that $R$ is a symmetrized subset of $F(X)$.
A nonempty word $b$ is called a {\it piece} if there exist distinct $w_1, w_2 \in R$
such that $w_1 \equiv bc_1$ and $w_2 \equiv bc_2$.
Let $D_1$ and $D_2$ be faces (not necessarily distinct) of $M$
with an edge $e \subseteq \partial D_1 \cap \partial D_2$.
Let $e \delta_1$ and $\delta_2e^{-1}$ be boundary cycles of $D_1$ and $D_2$, respectively.
Let $\phi(\delta_1)=f_1$ and $\phi(\delta_2)=f_2$. An $R$-diagram $M$
is said to be {\it reduced}
if one never has $f_2=f_1^{-1}$.
It should be noted that if $M$ is reduced
then $\phi(e)$ is a piece for every interior edge $e$ of $M$.

We recall the following lemma which is a well-known
classical result in combinatorial group theory (see \cite{lyndon_schupp}).

\begin{lemma} [van Kampen]
\label{lem:van_Kampen}
Suppose $G=\langle X \,|\, R \, \rangle$ with $R$ being symmetrized.
Let $v$ be a cyclically reduced
word in $X$. Then $v=1$ in $G$ if and only if
there exists a reduced van Kampen diagram $M$
over $G=\langle X \,|\, R \, \rangle$
with a boundary label $v$.
\end{lemma}

Let $M$ be an annular diagram over $G=\langle X \,|\, R \, \rangle$,
and let $K$ and $H$ be, respectively,
the unbounded and bounded components of $\RR^2-M$.
We call $\partial K (\subset \partial M)$
the {\it outer boundary} of $M$, while
$\partial H (\subset \partial M)$
is called the {\it inner boundary} of $M$.
Clearly, the {\it boundary} of $M$, $\partial M$, is the union of the outer boundary and the inner boundary.
A cycle of minimal length which contains all the edges in the outer (inner, resp.)
boundary of $M$ going around once along the boundary of $K$ ($H$, resp.)
is an {\it outer {\rm (}inner, {\rm resp.)} boundary cycle} of $M$.
An {\it outer {\rm (}inner, {\rm resp.)} boundary label of $M$} is defined to be a word $\phi(\alpha)$ in $X$
for $\alpha$ an outer (inner, resp.) boundary cycle of $M$.
The annular diagram $M$ is said to be {\it nontrivial}
if it contains a $2$-cell.

We recall another well-known classical result in combinatorial group theory.

\begin{lemma} [{\cite[Lemmas~V.5.1 and V.5.2]{lyndon_schupp}}]
\label{lem:lyndon_schupp}
Suppose $G=\langle X \,|\, R \, \rangle$ with $R$ being symmetrized.
Let $u, v$ be two cyclically reduced words in $X$
which are not trivial in $G$ and which are not conjugate in $F(X)$.
Then $u$ and $v$ represent conjugate elements in $G$ if and only if
there exists a reduced nontrivial annular diagram $M$ over $G=\langle X \,|\, R \, \rangle$
such that $u$ is an outer boundary label and $v^{-1}$ is an inner boundary label of $M$.
\end{lemma}

\section{Proof of Theorem~\ref{thm:main_theorem}}
\label{sec:proof_of_main_theorem}

\subsection{Proof of Theorem~\ref{thm:main_theorem}(1)}

Suppose on the contrary that there exists a rational number $s \in [1/n,1] \cup \{0\}$
such that for which $\alpha_s$ is null-homotopic in $\orbs(r;n)$.
Then $u_s$ equals the identity in $\Hecke(0;n)=\langle a,b \, |\, (ab)^n\rangle$.
Since $u_0$ is a non-trivial torsion element in
$\Hecke(0;n)$ by \cite[Theorem~IV.5.2]{lyndon_schupp},
we may assume $s \in [1/n,1]$.
By Lemma~\ref{lem:van_Kampen},
there is a reduced connected and simply-connected
diagram $M$ over $\Hecke(0;n)=\langle a, b \svert (ab)^n \rangle$ with
$(\phi(\partial M))=(u_s)$.

\medskip
{\bf Claim.} {\it There is no interior edge in $M$.}

\begin{proof}[Proof of Claim]
Suppose on the contrary that there are two $2$-cells $D_1$ and $D_2$ in $M$
such that $D_1$ and $D_2$ have a common edge $e$.
Since $M$ is reduced, $\phi(e)$ is a piece for the symmetrized subset $R$ of $F(a,b)$
generated by $\{(ab)^n\}$. But this is a contradiction,
since there is no piece for this $R$.
\end{proof}

Choose an extremal disk, say $J$, of $M$.
Here, recall that an {\it extremal disk} of a map $M$
is a submap of $M$ which is topologically a disk
and which has a boundary cycle $e_1, \dots, e_t$
such that the edges $e_1, \dots, e_t$ occur in order
in some boundary cycle of the whole map $M$.
Then by Claim, $J$ consists of only one $2$-cell.
This implies that $CS(\phi(\partial J))=\lp 2n \rp$,
so that $CS(\phi(\partial M))=CS(u_s)=CS(s)$ contains a term greater than or equal to
$2n$, which is a contradiction to Corollary~\ref{cor:properties}.
\qed

\begin{remark}
\rm
{Theorem~\ref{thm:main_theorem}(1) can be also proved
by using Newman's Spelling Theorem \cite[Theorem~3]{Newman}
(cf. \cite[Theorem~IV.5.5]{lyndon_schupp}),
which is a powerful theorem
for the word problem for one relator groups with torsion.
This implies that if a cyclically reduced word $v$
represents the trivial element in
$\Hecke(0;n)\cong \langle a,b \, |\, (ab)^n\rangle$,
then the cyclic word $(v)$
contains a subword of the cyclic word $((ab)^{\pm n})$
of length greater than $(n-1)/n=1-1/n$ times the length of $(ab)^n$,
so $(v)$ contains a subword $w$ of $((ab)^{\pm n})$
such that $|w|>2n(1-1/n)=2n-2 \ge n$.
Hence if $u_s=1$ in $\Hecke(0;n)$ for some $s \in [1/n,1]$,
then $CS(u_s)=CS(s)$ contains a term bigger than $n$,
which is a contradiction to Corollary~\ref{cor:properties}.
}
\end{remark}

\subsection{Proof of Theorem~\ref{thm:main_theorem}(2)}

Suppose on the contrary that there exist two distinct rational numbers $s$ and $s'$
in $[1/n,1]$ for which the simple loops $\alpha_s$ and $\alpha_{s'}$
are homotopic in $\orbs(0;n)$.
Then $u_s$ and $u_{s'}^{\pm 1}$ are conjugate in $\Hecke(0;n)$.
By Lemma~\ref{lem:lyndon_schupp}, there is a reduced nontrivial annular
diagram $M$ over $\Hecke(0;n)=\langle a, b \svert (ab)^n \rangle$ with
$(\phi(\alpha)) \equiv (u_s)$ and $(\phi(\beta)) \equiv (u_{s'}^{\pm 1})$,
where $\alpha$ and $\beta$ are, respectively,
the outer and inner boundary cycles of $M$.
Let the outer and inner boundaries of $M$ be denoted by
$\sigma$ and $\tau$, respectively.

\medskip
{\bf Claim 1.} {\it $\sigma$ and $\tau$ are simple,
i.e., they are homeomorphic to the circle.}

\begin{proof}[Proof of Claim 1]
Suppose on the contrary that $\sigma$ or $\tau$ is not simple.
Then there is an extremal disk, say $J$, of $M$.
As in the proof of Theorem~\ref{thm:main_theorem}(1),
$J$ consists of only one $2$-cell.
Then $CS(u_s)=CS(s)$ or $CS(u_{s'})=CS(s')$ contains a term greater than or equal to
$2n$, which is a contradiction to Corollary~\ref{cor:properties}.
\end{proof}

\medskip
{\bf Claim 2.} {\it $\sigma$ and $\tau$ do not have a common edge.}

\begin{proof}[Proof of Claim 2]
Suppose on the contrary that $\sigma$ and $\tau$ have a common edge $e$
as in Figure~{\rm\ref{fig.layer}(a)}.
Since $\sigma$ and $\tau$ are simple by Claim 1,
and since there is no interior edge in $M$
as in the proof of Theorem~\ref{thm:main_theorem}(1),
there is a vertex $v \in \sigma \cap \tau$ such that $d_M(v)=3$.
But since both $(u_s)$ and $(u_{s'})$ are alternating
and since $((ab)^n)$ is alternating,
this is a contradiction.
\end{proof}

\begin{figure}[h]
\begin{center}
\includegraphics{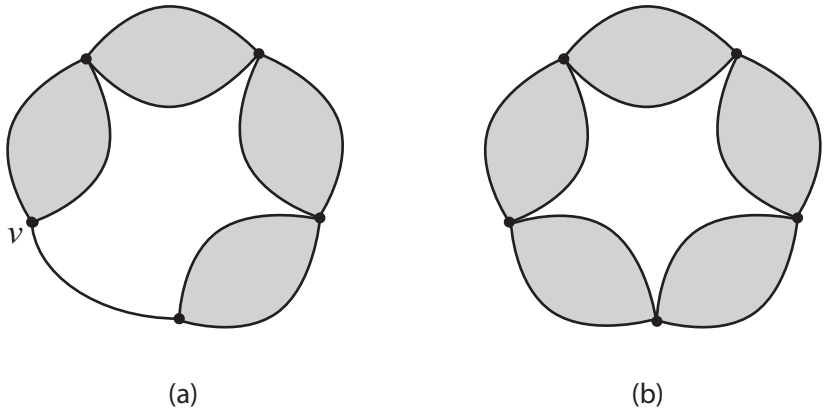}
\end{center}
\caption{\label{fig.layer}
(a) For the proof of Claim 2;
(b) A possible shape of $M$}
\end{figure}

By Claims 1--2 together with the fact that there is no interior edge in $M$
as in the proof of Theorem~\ref{thm:main_theorem}(1),
we see that Figure~\rm\ref{fig.layer}(b) illustrates the only possible shape of $M$.
In particular, $\sigma\cap\tau$ consists of finitely many vertices,
$M$ consists of a single layer, and the number of faces of $M$ is equal to
the number of degree $4$ vertices of $M$.
Here the number of faces is variable.

\begin{notation}
\label{notation:cell_boundary}
{\rm
Suppose that $M$ is a connected annular map as in Figure~\ref{fig.layer}(b).
Choose a vertex, say $v_0$, lying in both the outer and inner boundaries of $M$,
and let $\alpha$ and $\beta$ be, respectively,
the outer and inner boundary cycles of $M$ starting from $v_0$,
where $\alpha$ is read clockwise and $\beta$ is read counterclockwise.
Let $D_1,\dots, D_t$ be the $2$-cells of $M$
such that $\alpha$ goes through their boundaries in this order.
By the symbol $\partial D_i^{\pm}$, we denote an oriented edge path
contained in $\partial D_i$ such that
\[
\begin{aligned}
\alpha&=\partial D_1^+\cdots \partial D_t^+, \\
\beta^{-1}&=\partial D_1^-\cdots \partial D_t^-.
\end{aligned}
\]
}
\end{notation}

Then every $2$-cell $D$ of $M$ satisfies that
$\phi(\partial D^+)$ is a subword of the cyclic word $(\phi(\alpha))=(u_s)$
and that $\phi(\partial D^-)$ is a subword of the cyclic word $(\phi(\beta^{-1}))=(u_{s'}^{\pm 1})$.
Since $s, s' \in [1/n,1]$, every term of both $CS(s)$ and $CS(s')$ is less than or equal to $n$ by Corollary~\ref{cor:properties}.
Furthermore since $s \neq s'$, at least one of $CS(s)$ and $CS(s')$ has a term less than $n$.
Without loss of generality, we assume that $CS(s)$ has a term less than $n$.
This yields that there is a $2$-cell $D$ of $M$ such that
$\phi(\partial D^+)$ has length less than $n$.
But then $\phi(\partial D^-)$ has length bigger than $n$,
since $(\phi(\partial D^+)\phi(\partial D^-)^{-1})=((ab)^{\pm n})$.
This implies that $CS(s')$ contains a term bigger than $n$,
which is a contradiction to Corollary~\ref{cor:properties}.
\qed

\subsection{Proof of Theorem~\ref{thm:main_theorem}(3)}

Suppose on the contrary that there exists a rational number $s$
in $[1/n,1]$ for which the simple loop $\alpha_s$ is peripheral in $\orbs(0;n)$.
Then $u_s$ is conjugate to $a^{\pm t}$ or $b^{\pm t}$ in $\Hecke(1/p;n)$ for some integer $t \ge 1$.
We assume that $u_s$ is conjugate to $a^{\pm t}$ in $\Hecke(1/p;n)$.
(The case when $u_s$ is conjugate to $b^{\pm t}$ in $\Hecke(1/p;n)$ is treated similarly.)
By Lemma~\ref{lem:lyndon_schupp}, there is a reduced nontrivial annular
diagram $M$ over $\Hecke(0;n)=\langle a, b \svert (ab)^n \rangle$ with
$(\phi(\alpha)) \equiv (u_s)$ and $(\phi(\beta)) \equiv (a^{\pm t})$,
where $\alpha$ and $\beta$ are, respectively,
the outer and inner boundary cycles of $M$.

Let the outer and inner boundaries of $M$ be denoted by
$\sigma$ and $\tau$, respectively.
As in Claim~1 in the proof of Theorem~\ref{thm:main_theorem}(2),
$\sigma$ and $\tau$ are simple.
However Claim~2 in the proof of Theorem~\ref{thm:main_theorem}(2)
does not hold, because $(u_s)$ and $((ab)^n)$ are alternating
while $(a^{\pm t})$ is not.
So $\sigma$ and $\tau$ might have a common edge,
and hence $M$ can be shaped as in Figure~\rm\ref{fig.layer}(a)
and (b).
In either case, every $2$-cell $D$ satisfies that
$\phi(\partial D^+)$ is a subword of the cyclic word $(u_s)$
and that $\phi(\partial D^-)$ is a subword of the cyclic word $(a^{\pm t})$.
Here, the only possibility is that $\phi(\partial D^-)$ has length $1$,
since $\phi(\partial D^-)$ is also a subword of the cyclic word $((ab)^{\pm n})$.
But then $\phi(\partial D^+)$ has length $2n-1$, which implies that
$CS(\phi(\alpha))=CS(s)$ contains a term bigger than or equal to
$2n-1>n$. This is a contradiction to Corollary~\ref{cor:properties}.
\qed

\subsection{Proof of Theorem~\ref{thm:main_theorem}(4)}
Suppose that there exists a rational number $s$ in $[1/n,1]$
for which the simple loop $\alpha_s$ is torsion in $\orbs(0;n)$.
Then
$u_s^t=1$ in $\Hecke(0;n)=\langle a,b \, |\, (ab)^n\rangle$
for some integer $t \ge 1$.
By Lemma~\ref{lem:van_Kampen},
there is a reduced connected and simply-connected
diagram $M$ over $\Hecke(0;n)=\langle a, b \svert (ab)^n \rangle$ with
$(\phi(\partial M))=(u_s^t)$.
Choose an extremal disk, say $J$, of $M$
Since there is no interior edge in $M$
as in the proof of Theorem~\ref{thm:main_theorem}(1),
$J$ consists of only one $2$-cell.
But then $CS(\phi(\partial M))=CS(u_s^t)$ contains a term greater than or equal to
$2n$, which is a contradiction because every term of $CS(u_s)$, so of $CS(u_s^t)$,
is less than or equal to $n$.
\qed

\bibstyle{plain}

\end{document}